\documentclass{amsart}

\newcommand{\C}{\mathcal C}
\newcommand{\F}{\mathcal F}
\newcommand{\U}{\mathcal U}
\newcommand{\V}{\mathcal V}
\newcommand{\St}{\mathcal{S}t}
\newcommand{\e}{\varepsilon}
\newcommand{\Ra}{\Rightarrow}
\newcommand{\IN}{\mathbb N}

\newtheorem{problem}{Problem}
\newtheorem{theorem}{Theorem}
\newtheorem{claim}{Claim}

\title{A parallel metrization theorem}
\author{Taras Banakh and Olena Hryniv}
\address{Jan Kochanowski University in Kielce (Poland) and Ivan Franko National University of Lviv (Ukraine)}
\email{t.o.banakh@gmail.com, ohryniv@gmail.com}

\begin{document}
\begin{abstract} Two non-empty sets $A,B$ of a metric space $(X,d)$ are called  parallel if $d(a,B)=d(A,B)=d(A,b)$ for any points $a\in A$ and $b\in B$. Answering a question posed on {\tt mathoverflow.net}, we prove that for a cover $\C$ of a metrizable space $X$ the following conditions are equivalent: (i) the topology of $X$ is generated by a metric $d$ such that any two sets $A,B\in\C$ are parallel; (ii) the cover $\C$ is disjoint, lower semicontinuous and upper semicontinuous.
\end{abstract}
\maketitle

In this paper we shall prove a ``parallel'' metrization theorem answering a question \cite{MO} of the Mathoverflow {\tt user116515}. The question concerns parallel sets in metric spaces.

Two non-empty sets $A,B$ in a metric space $(X,d)$ are called {\em parallel} if $$d(a,B)=d(A,B)=d(A,b)\mbox{ \ \ for any $a\in A$ and $b\in B$}.$$ Here $d(A,B)=\inf\{d(a,b):a\in A,\; b\in B\}$ and $d(x,B)=d(B,x):=d(\{x\},B)$ for $x\in X$.
Observe that two closed parallel sets $A,B$ is a metric space are either disjoint or coincide.

Let $\C$ be a family of non-empty closed subsets of a topological space $X$. A metric $d$ on $X$ is defined to be {\em $\C$-parallel} if any two sets $A,B\in\C$ are parallel with respect to the metric $d$.

A family $\C$ of subsets of $X$ is called a {\em compact cover} of $X$ if $X=\bigcup\C$ and each set $C\in\C$ is compact.

In this paper we shall consider the following problem posed on MathOverflow \cite{MO}.

\begin{problem}\label{prob1} For which compact covers $\C$ of a topological space $X$ the topology of $X$ is generated by a $\C$-parallel metric?
\end{problem}

A metric generating the topology of a given topological space will be called {\em admissible}.

A necessary condition for the existence of an admissible $\C$-parallel metric is the upper and lower semicontinuity of the cover $\C$.

A family $\C$ of subsets of a topological space $X$ is called
\begin{itemize}
\item {\em lower semicontinuous} if for any open set $U\subset X$ its $\C$-star $St(U;\C):=\bigcup\{C\in\C:C\cap U\ne\emptyset\}$ is open in $X$;
\item {\em upper semicontinuous} if for closed set $F\subset X$ its $\C$-star $St(F;\C)$ is closed in $X$;
\item {\em continuous} if $\C$ is both lower and upper semicontinuous;
\item {\em disjoint} if any distinct sets $A,B\in\C$ are disjoint.
\end{itemize}

The following theorem is the main result of the paper, answering Problem~\ref{prob1}.

\begin{theorem} For a compact cover $\C$ of a metrizable topological space $X$ the following conditions are equivalent:
\begin{enumerate}
\item the topology of $X$ is generated by a $\C$-parallel metric;
\item the family $\C$ is disjoint and continuous.
\end{enumerate}
\end{theorem}

\begin{proof} $(1)\Ra(2)$ Assume that $d$ is an admissible $\C$-parallel metric on $X$. The disjointness of the cover $\C$ follows from the obvious observation that two closed parallel sets in a metric space are
either disjoint or coincide.

To see that $\C$ is lower semicontinuous, fix any open set $U\subset X$ and consider its $\C$-star $\St(U;\C)$. To see that $\St(U;\C)$ is open, take any point $s\in \St(U;\C)$ and find a set $C\in\C$ such that $s\in C$ and $C\cap U\ne\emptyset$. Fix a point $u\in U\cap C$ and find $\e>0$ such that the $\e$-ball $B(u,\e)=\{x\in X:d(x,u)<\e\}$ is contained in $U$. We claim that $B(s,\e)\subset \St(U;\C)$.
Indeed, for any $x\in B(s,\e)$ we can find a set $C_x\in\C$ containing $x$ and conclude that $d(C_x,u)=d(C_x,C)\le d(x,s)<\e$ and hence $C_x\cap U\ne\emptyset$ and $x\in C_x\subset\St(U,\C)$.

To see that $\C$ is lower semicontinuous, fix any closed set $F\subset X$ and  consider its $\C$-star $\St(F;\C)$. To see that $\St(F;\C)$ is closed, take any point $s\in X\setminus \St(F;\C)$ and find a set $C\in\C$ such that $s\in C$. It follows from $s\notin \St(F,\C)$ that $C\cap F=\emptyset$ and hence $\e:=d(C,F)>0$ by the compactness of $C$. We claim that $B(s,\e)\cap\St(F,\C)=\emptyset$. Assuming the opposite, we can find a point $x\in B(s,\e)\cap St(F,\C)$ and a set $C_x\in\C$ such that $x\in C_x$ and $C_x\cap F\ne\emptyset$. Fix a point $z\in C_x\cap F$ and observe that $d(C,F)\le d(C,z)=d(C,C_x)\le d(s,x)<\e=d(C,F)$, which is a desired contradiction. 
\smallskip

The proof of the implication $(2)\Ra(1)$ is more difficult. Assume that $\C$ is disjoint and continuous. Fix any admissible metric $\rho\le 1$ on $X$.

Let $\U_0(C)=\{X\}$ for every $C\in\C$.

\begin{claim} For every $n\in\IN$ and every $C\in\C$ there exists a finite cover $\U_n(C)$ of $C$ by open subsets of $X$ such that
\begin{enumerate}
\item[(i)] each set $U\in\U_n(C)$ has $\rho$-diameter $\le\frac1{2^n}$;
\item[(ii)] if a set $A\in\C$ meets some set $U\in\U_n(C)$, then $A\subset\bigcup\U_n(C)$ and $A$ meets each set $U'\in\U_n(C)$.
\end{enumerate}
\end{claim}

\begin{proof} Using the paracompactness  \cite[5.1.3]{En} of the metrizable space $X$, choose an open locally finite cover $\V$ of $X$ consisting of sets of $\rho$-diameter $<\frac1{2^n}$.

For every compact set $C\in\C$ consider the finite subfamily $\V(C):=\{V\in\V:V\cap C\ne \emptyset\}$ of the locally finite cover $\V$. Since the cover $\C$ is upper semicontinuous, the set $F_C=\St(X\setminus \bigcup\V(C);\C)$ is closed and disjoint with the set $C$. Since $\C$ is lower semi-continuous, for any open set $V\in\V(C)$ the set $\St(V;\C)$ is open and hence  $W(C):=\bigcap_{V\in\V(C)}\St(V;\C)\setminus F_C$ is an open neighborhood of $C$.

Put $\U_n(C):=\{W(C)\cap V:V\in\V(C)\}$ and observe that $\U_n$ satisfies the condition (i).

Let us show that the cover $\U_n(C)$ satisfies the condition (ii). Assume that a set $A\in\C$ meets some set $U\in\U_n(C)$. First we show that $A\subset \bigcup\U_n(C)$. Find a set $V\in\V(C)$ such that $U=W(C)\cap V$. It follows that $\emptyset\ne A\cap U\subset A\cap W(C)$ that the set $A$ meets $W(C)$ and hence is contained in $W(C)$ and is disjoint with $F_C$. Hence $$A\subset W(C)\cap\big({\textstyle{\bigcup}\,\V(C)}\big)=\bigcup\limits_{V\in\V(C)}W(C)\cap V={\textstyle{\bigcup}\,\U_n(C)}.$$

Next, take any set $U'\in\U_n(C)$ and find a set $V'\in\V(C)$ with $U'=W(C)\cap V'$. The relation $A\cap W(C)\cap V=A\cap U\ne\emptyset$ and the definition of the set $W(C)\supset A$ implies that $A$ intersects the set $V'\in\V(C)$ and hence intersects the set $U'=W(C)\cap V'$. This completes the proof of Claim.
\end{proof}

Given two points $x,y\in X$ let $$\delta(x,y):=\inf\big\{\tfrac1{2^n}:\mbox{$\exists C\in\C$ and $U\in\U_n(C)$ such that $(x,y)\in U$}\big\}.$$
Adjust the function $\delta$ to a pseudometric $d$ letting $$d(x,y)=\inf\sum_{i=1}^m\delta(x_{i-1},x_i)$$where the infimum is taken over all sequences $x=x_0,\dots,x_m=y$.
The condition (i) of Claim implies that $\rho(x,y)\le\delta(x,y)$ and hence $\rho(x,y)\le d(x,y)$ for any $x,y\in X$. So, the pseudometric $d$ is a metric on $X$ such that the identity map $(X,d)\to (X,\rho)$ is continuous. To see that this map is a homeomorphism, take any point $x\in X$ and $\e>0$. Find $n\in\IN$ such that $\frac1{2^n}<\e$ and choose a set $C\in\C$ with $x\in C$ and a set $U\in\U_n(C)$ with $x\in U$. Then for any $y\in U$ we get $d(y,x)\le\delta(x,y)\le\frac1{2^n}<\e$, which means that the map $X\to (X,d)$ is continuous.

Finally, let us prove that the metric $d$ is $\C$-parallel. Pick any two distinct compact sets $A,B\in\C$. We need to show that $d(a,B)=d(A,B)=d(A,b)$ for any $a\in A$, $b\in B$. Assuming that this inequality is not true, we conclude that either $d(a,B)>d(A,B)>0$ or $d(A,b)>d(A,B)>0$ for some $a\in A$ and $b\in B$.

First assume that $d(a,B)>d(A,B)$ for some $a\in A$. Choose points $a'\in A$, $b'\in B'$ such that $d(a',b')=d(A,B)<d(a,B)$. By the definition of the distance $d(a',b')<d(a,B)$, there exists a chain $a'=x'_0,x'_1,\dots,x'_m=b'$ such that $\sum_{i=1}^m\delta(x'_{i-1},x'_i)<d(a,B)$. We can assume that the points $x'_0,\dots,x'_m$ are pairwise distinct, so for every $i\le m$ there exist $n_i\ge 0$ such that $\delta(x'_{i-1},x'_i)=\frac1{2^{n_i}}$ and hence  $x'_{i-1},x'_i\in U_i'$ for some $C_i\in\C$ and $U_i'\in\U_{n_i}(C_i)$. For every $i\le m$ let $A_i\in\C$ be the unique set with $x_i'\in A_i$. Then $A_0=A$ and $A_m=B$.

Using the condition (ii), we can inductively construct a sequence of points $a=x_0,x_1,\dots,x_m\in B$ such that for every positive $i\le m$ the point $x_i$ belongs to $A_i$ and the points $x_{i-1},x_i$ belong to some set $U_i\in\U_{n_i}(C_i)$. The chain $a=x_0,x_1,\dots,x_m\in A_m=B$ witnesses that $$d(a,B)\le  d(a,x_m)\le\sum_{i=1}^m\delta(x_{i-1},x_i)\le\sum_{i=1}^m\tfrac1{2^{n_i}}=
\sum_{i=1}^m\delta(x'_{i-1},x_i')<d(a,B),$$
which is a desired contradiction.

By analogy we can prove that the case $d(A,B)<d(A,b)$ leads to a contradiction.

\end{proof}

\end{document}